\documentclass[12pt]{article}
\usepackage[utf8]{inputenc}
\usepackage{amssymb,amsmath,amsfonts,amsthm,amscd,latexsym,indentfirst,verbatim,xcolor}
\usepackage[T1]{fontenc}
\usepackage{geometry}
\def\rb{\textrm{rb}}
\def\RB{\textrm{RB}}
\def\dbl{\lbrace\kern-3pt\lbrace}
\def\dbr{\rbrace\kern-3pt\rbrace}

\def\diag{\textrm{diag}}
\def\Imm{\textrm{Im}}

\def\Aut{\textrm{Aut}}

\def\Span{\textrm{Span}}
\def\id{\operatorname{id}}
\def\ad{\operatorname{ad}}

\def\Imm{\mathrm{Im}\,}
\def\charr{\mathrm{char}\,}

\theoremstyle{plain}
\newtheorem{theorem}{Theorem}[section]
\newtheorem{lemma}[theorem]{Lemma}

\newtheorem*{conjecture*}{Conjecture}

\theoremstyle{definition}

\begin{document}
\sloppy
\hfill{16W99, 16S50 (MSC2020)}
\begin{center}
{\Large Extremal Rota---Baxter operators on matrix algebras}
\smallskip

Vsevolod Gubarev
\end{center}

\begin{abstract}
We classify all Rota---Baxter operators of weight~$\lambda$ on $M_n(F)$ over a~field~$F$ of characteristic zero
whose minimal polynomial has maximal degree and show that in both cases $\lambda = 0$ and $\lambda \neq 0$
such an operator is conjugate to exactly one operator.
\end{abstract}

\section{Introduction}

Let $A$ be an algebra over a field~$F$. A linear map $R\colon A\to A$ is a
Rota---Baxter operator of weight $\lambda\in F$ if
\begin{equation}\label{RB}
R(x)R(y) = R(R(x)y+xR(y)+\lambda xy)
\end{equation}
holds for all $x,y \in A$.

The notion of Rota---Baxter operators was introduced by G. Baxter in 1960~\cite{Baxter} in connection with a~certain problem in analysis.
Further, G.-C. Rota~\cite{Rota}, his coauthors and students continued the study of such operators.

Nowadays, Rota---Baxter operators are widely known because of their fruitful connections, including
many versions of Yang---Baxter equation, pre- and postalgebras, skew braces, double Lie algebras, decompositions of algebras, etc., see~\cite{Guo}.

Rota---Baxter operators on the matrix algebra play a special role.
On $M_n(\mathbb{C})$, solutions of the associative Yang---Baxter equation~\cite{Aguiar,Polishchuk,Zhelyabin} are in one-to-one correspondence with Rota---Baxter operators of weight zero~\cite[Theorem\,3.4]{GubarevUnital}.
Moreover, solutions of the weighted associative Yang---Baxter equation~\cite{FardThesis,OgievetskyPopov} on $M_n(\mathbb{C})$ are in one-to-one correspondence with Rota---Baxter operators of nonzero weight~\cite[Theorem 4.12]{AYBE-ext}.

Also, skew-symmetric Rota---Baxter operators of weight zero on~$M_n(\mathbb{C})$
are in one-to-one correspondence with finite-dimensional double Lie algebras~\cite{DoubleLie,DoublePoissonFree,Schedler} arising in noncommutative geometry.

Rota---Baxter operators of weight zero on $M_2(F)$ were classified in~\cite{Aguiar,Mat2,Panasenko2} and of nonzero weight in~\cite{BGP,GoncharovGubarevM3}.
Rota---Baxter operators of nonzero weight were described on $M_3(F)$ in the series of works~\cite{GoncharovGubarevM3,Gub2021,Gub2024}, while the zero weight case was done in~\cite{GubarevM3Zero} under the condition that the identity matrix does not lie in the kernel of the operator, see also the classification of all skew-symmetric Rota---Baxter operators of weight zero on $M_3(\mathbb{C})$~\cite{Sokolov} in terms of solutions to the associative Yang---Baxter equation.

However, the general classification problem of all Rota---Baxter operators on $M_n(F)$ looks wild.
Thus, we need to find reasonable and important classes of them.
The current work is devoted to the classification of all extremal Rota---Baxter operators on $M_n(F)$ over a field of characteristic zero, that is, operators whose minimal polynomial has maximal degree, which has to be equal to~$2n-1$.

Let us write down such Rota---Baxter operators on $M_n(F)$.
For the case of weight zero, we have $B_0\colon M_n(F)\to M_n(F)$ defined by
\begin{equation}\label{eq:B0}
B_0(e_{ij}) = \begin{cases}
 \sum\limits_{r\geq0}e_{i+r,j+r+1},  & i\leq j,\\
 -\sum\limits_{r\geq1}e_{i-r,j-r+1}, & i>j.
\end{cases}
\end{equation}

For the case of weight one, we have $B_1\colon M_n(F)\to M_n(F)$ defined by
\begin{equation}\label{eq:B1}
B_1(e_{ij}) = \begin{cases}
 \sum\limits_{r\geq1}e_{i+r,j+r},& i\geq j,\\
 -\sum\limits_{r\geq0}e_{i-r,j-r},& i<j.
\end{cases}
\end{equation}

The Rota---Baxter operators $B_0$ and $B_1$ have already appeared in quite different contexts several times.
They were written down by the author in~\cite[Example 5.18]{GubarevUnital} and~\cite[Example 1]{GubarevSpectrum} exactly as examples of Rota---Baxter operators providing the maximal Rota---Baxter index, it means having the maximal possible degree of their minimal polynomial.

Earlier, the operator $B_0$ was written down in terms of the solution of the associative Yang---Baxter equation
for $n = 2$ by M. Aguiar in 2000~\cite[Example 2.3.3]{Aguiar} and the conjugate operator for $n = 3$ by V.~Sokolov~\cite[Theorem~1]{Sokolov} in 2013.
The operator~$B_1$ appeared in~\cite{BGP} for $n = 2$ and in~\cite{GoncharovGubarevM3} for $n = 3$ within the classifications of Rota---Baxter operators of nonzero weight on the matrix algebras of small order.

In 2010, O.~Ogievetsky and T.~Popov obtained the operators $B_0$ and $B_1$ from the B\'ezout operator
in the context of the associative and weighted associative Yang---Baxter equation~\cite[Eqs.\,(5.43),\,(5.44)]{OgievetskyPopov}.

The prolongations of the operators~$B_0$ and $B_1$ on the infinite matrices served as a~source for the first and, as far as we know, unique examples of simple double Lie algebras and simple double Lie algebras of weight~1, respectively, see~\cite[Theorem 3]{GubarevDouble} and~\cite[Theorems\,3.3,\,4.3]{GubarevLaurent}.
Double Lie algebras of nonzero weight were introduced in~\cite{GoncharovGubarev}.

Now, we formulate the main result of the paper.

\begin{theorem}\label{thm:main}
Let $F$ be a field of characteristic zero and let $R$ be a Rota---Baxter operator of weight~$\lambda$ on $M_n(F)$ whose minimal polynomial has degree~$2n-1$. 

a) If $\lambda = 0$, then $R$ up to conjugation with automorphisms and transpose and up to nonzero scalar multiplication equals $B_0$.

b) If $\lambda \neq 0$, then $R$ up to conjugation with automorphisms and transpose and up to action of~$\phi$ equals $\lambda B_1$.
\end{theorem}

It is worth noting that, up to equivalence, there exists a~unique Rota---Baxter operator of weight~1 on the split Cayley---Dickson algebra, the 8-dimensional simple alternative algebra, whose minimal polynomial has maximal degree~\cite{Panasenko}.

Let us comment on the proof of Theorem~\ref{thm:main}.
In the case of weight zero, maximal nilpotency forces $\Imm R$ to be a maximal linear space of singular matrices (Lemma~\ref{lem:zero-image}). 
Hence, the inverse of $R$ becomes a derivation which is inner in general (Lemma~\ref{lem:inner-derivation}).
Finally, the condition on the image of the identity fixes~$R$ up to the required equivalence.

In the case of nonzero weight, we show that, up to some conjugation, $\ker R^n$ consists of lower-triangular matrices and $\ker (R+\id)^n$ consists of strictly upper-triangular matrices (Lemma~\ref{lem:triangular}).
Endomorphisms $R(R+\id)^{-1}$ and $(R+\id)R^{-1}$ of $\ker R^n$ and $\ker (R+\id)^n$, respectively, recover $R$ as the geometric sum of two opposite shifts (Lemma~\ref{lem:shifts}).

\section{AI usage}

Let us clarify the author's interest in the problem and the AI usage for the presented proof.
In summer of 2020, the author formulated to himself a problem of description of all Rota---Baxter operators on $M_n(F)$ whose minimal polynomial has maximal degree. 
By that moment, the author knew the classification of such Rota---Baxter operators of any weight on $M_3(F)$ and that Theorem~\ref{thm:main} holds for $n\leq 3$.
Further, this led the author to conjecture that such extremal Rota---Baxter operators have only one orbit for a~given weight.
In particular, the author reported this conjecture in his talk given at the Shirshov seminar (Sobolev Institute of Mathematics) on 30 October 2024.
Then it was checked with the help of the computer algebra system \texttt{Singular} that Theorem~\ref{thm:main} holds when $\lambda = 0$ and $n = 4$.
However, the strategy of involving \texttt{Singular} to guess the corresponding conjugate matrix in the general case looked not so clear and easy.

The author's idea for the proof of Theorem~\ref{thm:main} when $\lambda\neq0$ was the following.
First, describe all Rota---Baxter finite-dimensional modules over $F^n$ with $R$ satisfying 
$R(e_i) = e_{i+1} + \ldots + e_n$ for $i<n$ and $R(e_n) = 0$ (due to Lemma~\ref{lem:diag-known}). 
Second, combine such actions of the diagonal subalgebra on all diagonal subspaces $V_d = \Span\{e_{ij} \mid i-j =d\}$ and try to get the required proof.

In August 2026, the author used  Aristotle~\cite{Aristotle} to generate candidate proofs of a) and b) of Theorem~\ref{thm:main}.
The suggested proofs were actually long and technical.
After some time, the author used ChatGPT 5.6 Sol in attempt to find a canonical (due to T. Tao's concept~\cite{Tao}) proof of this result from the Book.
After several attempts, ChatGPT 5.6 Sol found a simpler proof of Theorem~\ref{thm:main}a and clarified Aristotle's proof of Theorem~\ref{thm:main}b.
However, the prepared proof still required substantial rewriting and splitting into lemmas.
So, the author adapted, reorganized, and verified the proof.
The author is responsible for the content of the paper.

For the case of weight zero, an elegant linear algebra fixes the required image subalgebra and the kernel subspace (Lemma~\ref{lem:zero-image}).
The author believes that Lemma~\ref{lem:inner-derivation} follows from some deep results but the presented simple proof better aligns with the concept of a~proof from the Book.

For the case of nonzero weight, the very simple Lemma~\ref{lem:hom-known} occurs to be useful; and again the image of the identity plays a crucial role, compare with~\cite{GoncharovGubarevM3,GubarevUnital,GubarevSpectrum}.

\section{Preliminaries}

Throughout the paper, $F$ is a field of characteristic zero.

For every Rota---Baxter operator $R$ of weight~$\lambda$ on an algebra~$A$, we have that $\Imm R$ is a subalgebra of~$A$.
If $\lambda\neq0$, then $\Imm (R+\lambda\id)$ is a subalgebra of~$A$ as well. If $\lambda = 0$, then $\ker R$ is an $\Imm R$-bimodule, it means that $\Imm R\cdot \ker R,\ker R\cdot \Imm R\subseteq \ker R$.

\begin{lemma}[{\cite{Guo}}]
Given an RB-operator $R$~of weight~$\lambda$,

(a) the operator $-R-\lambda\id$ is an RB-operator of weight $\lambda$,

(b) the operator $\lambda^{-1}R$ is an RB-operator of weight 1, provided $\lambda\neq0$.
\end{lemma}

Given an algebra $A$, denote the set of all RB-operators of weight~$\lambda$ on~$A$ by $\RB_\lambda(A)$.
Define a map $\phi_\lambda\colon \RB_\lambda(A)\to \RB_\lambda(A)$ as $\phi(R)=-R-\lambda\id$.
Note that $\phi^2 = \id$.

\begin{lemma}[{\cite{BGP}}]
Given an algebra $A$, an RB-operator $R$ on $A$ of weight $\lambda$,
and $\psi\in\Aut(A)$, the operator $R^{(\psi)} = \psi^{-1}R\psi$
is an RB-operator of weight~$\lambda$ on~$A$.
\end{lemma}

The same result holds when $\psi$ is an antiautomorphism of $A$, i.\,e.,
an automorphism of~$A$, considered as a~vector space, satisfying $\psi(xy) = \psi(y)\psi(x)$ for all $x,y\in A$;
we will apply transpose on a~matrix algebra.

By~\cite{GubarevSpectrum,GubarevUnital}, the Rota---Baxter $\lambda$-index ($\RB(\lambda)$-index) of $A$ is defined as follows:
$$
\rb_\lambda(A) = \min\{n\in\mathbb{N}\mid \mbox{for all }R\in \RB_\lambda(A)\
 \mbox{exists } 0\leq k\leq n:\ R^k(R+\lambda\id)^{n-k} = 0\}.
$$
If such a number does not exist, then set $\rb_\lambda(A) = \infty$.

\begin{theorem}[{\cite{GubarevSpectrum}}] \label{theo:rb-index}
Let $F$ be a field of characteristic zero. Then
$\rb_\lambda(M_n(F)) = 2n-1$ for all $\lambda\in F$.
\end{theorem}

To prove the main result, we will apply the following facts for $\lambda = 0$.

\begin{lemma}[{\cite[Lemma 5.17a]{GubarevUnital}}] \label{lem:upper-bound}
Given a Rota---Baxter operator of weight~0 on $M_n(F)$,
every matrix in $\Imm R$ is singular and $\dim\Imm R\leq n^2-n$.
\end{lemma}

\begin{lemma}[{\cite[Theorem 3]{Meshulam}}] \label{lem:meshulam}
Let $W\subset M_n(F)$ be a linear subspace consisting of singular matrices.
If $\dim W = n^2-n$, then either all matrices in $W$ have a common nonzero kernel vector or
all their images are contained in a fixed hyperplane.
\end{lemma}

\begin{lemma}[{\cite[Theorem\,5.22,\,Lemma\,5.23]{GubarevUnital}}] \label{lem:zero-known}
Let $R$ be a Rota---Baxter operator of weight~0 on $M_n(F)$ such that $R^{2n-2}\neq0$.
Then, up to conjugation with an automorphism,
$R(1) = N := e_{12} + \ldots + e_{n-1,n}$ and
$R(A_d)\subset\bigoplus_{r>d}A_r$, where
$A_d = \Span\{e_{ij}\mid j-i=d\}$.
\end{lemma}

To prove the main result, we will apply the following facts for $\lambda \neq 0$.
By $m_R$ we mean the minimal polynomial of~$R$.

\begin{lemma}\label{lem:nz-known}
Let $R$ be a Rota---Baxter operator of weight~1 on $M_n(F)$ whose minimal polynomial
has degree $2n-1$. Then up to the action of~$\phi$
and up to conjugation with automorphism, one may assume

a) \cite[Corollary\,1,\,proof of Theorem\,4]{GubarevSpectrum} $m_R(X) = X^n(X+1)^{n-1}$.

b) \cite[Theorem\,4.17]{GubarevUnital}, \cite[Lemma\,3.4]{GoncharovGubarevM3} $R(1) = \diag\{-s,-s+1,\ldots,0,1,\ldots,r-1,r\}$
for some natural $r,s$ such that $r+s+1 = n$ and 
every space $V_d = \Span\{e_{ij} \mid i-j = d\}$ is $R$-invariant.

c) \cite[Lemma\,8b]{GubarevSpectrum} $M_n(F) = \ker R^n \oplus \ker(R+\id)^{n-1}$ is a direct vector space sum of two subalgebras.
\end{lemma}

By $D_n(F)$ we denote the subalgebra of diagonal matrices in~$M_n(F)$.

\begin{lemma}[{\cite{GubarevFields}\label{lem:diag-known}}]
Let $R$ be a Rota---Baxter operator of weight~1 on $M_n(F)$ such that $D_n(F)$ is $R$-invariant,
$(R|_{D_n(F)})^n = 0$, and $(R|_{D_n(F)})^{n-1}\neq0$.
Then we can conjugate~$R$ with an automorphism~$\psi$ such that 
$D_n(F)$ is $R^{(\psi)}$-invariant and $P = (R^{(\psi)})|_{D_n(F)}$ acts as follows:
$P(e_{ii}) = e_{i+1,i+1} + \ldots + e_{nn}$ for $1\leq i<n$ and $P(e_{nn}) = 0$.
\end{lemma}

Given an algebra~$A$, we denote the operator of the left multiplication by a fixed element~$a\in A$ as $L_a$.
By $\ad_a$ we denote the operator $\ad_a\colon x\to [a,x] = ax - xa$.

\begin{lemma}[{\cite{GubarevUnital}}]\label{lem:hom-known}
Let $R$ be a Rota---Baxter operator of weight~1 on an algebra~$A$ and put $a = R(1)$. Then

a) $[L_a,R] = R(R+\id)$ and $[a,R(x)] = R([a,x])$,

b) Suppose that $A = \ker R^n \oplus \ker(R+\id)^{n-1}$ is a direct vector space sum of subalgebras.
Then $\Phi_0 = R(R+\id)^{-1}$ is an endomorphism of $\ker R^n$ and
$\Phi_1 = (R+\id)R^{-1}$ is an endomorphism of $\ker(R+\id)^{n-1}$.
Moreover, $\Phi_0$ ($\Phi_1$) commutes on $\ker R^n$ ($\ker (R+\id)^{n-1}$) with $\ad_a$ and satisfies
$[L_a,\Phi_0] = \Phi_0$ ($[L_a,\Phi_1] = -\Phi_1$).
The subspaces $\ker R^n,\ker(R+\id)^{n-1}$ are $L_a$-invariant.
\end{lemma}

\begin{proof}
Actually, we only need to prove the last two parts of b.
The equalities involving $\Phi_0,\Phi_1$ follow from a) and $[L,T^{-1}] = -T^{-1}[L,T]T^{-1}$.
Finally, $[L_a,R^k] = k(R^{k+1}+R^k)$ and, for $S = R+\id$,
$[L_a,S^k] = k(S^{k+1}-S^k)$, which gives the $L_a$-invariance of $\ker R^n$ and $\ker(R+\id)^{n-1}$.
Indeed, take $x\in \ker R^n$, then $R^n(L_a(x)) = L_a R^n(x) - nR^n(R+\id)(x) = 0$, hence
$L_a(x) \in \ker R^n$. Analogously, we have $L_a$-invariance of $\ker(R+\id)^{n-1}$.
\end{proof}

\section{Weight zero}

The assertion for $n\leq2$ follows from the classification~\cite{Panasenko2}.
Hence throughout this section we assume $n\geq3$.

We first determine the image and the kernel of $R$.
Define
$$
B = \Span\{ e_{ij} \mid 1\leq i\leq n-1,\,1\leq j\leq n\}, \quad
K = \Span \{e_{in}\mid 1\leq i\leq n\}.
$$

\begin{lemma}\label{lem:zero-image}
Let $R$ satisfy the hypotheses of Theorem~\ref{thm:main}a and $R(1) = N$. 
Then there exists an (anti)automorphism~$\psi$ of $M_n(F)$ such that
$\Imm R^{(\psi)} = B$ and
$\ker R^{(\psi)} = K$.
\end{lemma}

\begin{proof}
Since every application of $R$ strictly raises the diagonal degree and
$R^{2n-2}\neq0$, necessarily $R^{2n-2}(e_{n1})\neq0$ and the component of $R(e_{n1})$ in $A_{-(n-2)}$ is nonzero.

Put $C=F[s,t]/(s^n,t^n)$ and define the $F$-linear map $\Theta\colon C\to M_n(F)$ by
$$
\Theta(s^ a t^b) = N^a e_{n1} N^b, \quad 0\leq a,b<n.
$$  
Since $N^a e_{n1} N^b = e_{n-a,b+1}$, $\Theta$ is an isomorphism of vector spaces. 
We also have $\Theta(s^a h t^b) = N^a\Theta(h)N^b$ for every $h\in C$ and for all $a,b\geq0$.
Define $g = \Theta^{-1}(R(e_{n1}))$ and so $R(e_{n1}) = \Theta(g)$. 
By Lemma~\ref{lem:zero-known}, $g$~has zero constant term, while the first paragraph shows that its linear part $\ell$~is nonzero.

Since $N,R(e_{n1})\in\Imm R$ and $\Imm R$ is a subalgebra, $\Theta(gC)\subset\Imm R$. 

Write $g = \ell + g_{\geq2}$, where $\ell = \alpha s + \beta t\neq0$. 
Then $\dim gC = \dim\ell C = n^2 - n$.
$$
n^2 - n
 = \dim \ell C
 \leq \dim gC
 = \dim \Theta(gC)
 \leq \dim\Imm R.
$$

Applying Lemma~\ref{lem:upper-bound}, we get $\dim\Imm R = n^2-n$. 
By Lemma~\ref{lem:meshulam}, either all matrices in $\Imm R$ have a common nonzero kernel
vector or all their images lie in a fixed hyperplane. 
Since $N = R(1)\in\Imm R$ and $Ne_1 = 0$, $Ne_i = e_{i-1}$ for $2\leq i\leq n$, the common kernel line in the first
case is $Fe_1$. In the second case, the common image hyperplane contains
$Ne_2,\ldots,Ne_n$, hence it contains $\Span\{e_1,\ldots,e_{n-1}\}$; since both spaces have dimension $n-1$, they are equal.
The antiautomorphism $\psi\colon e_{ij}\to e_{n+1-j,n+1-i}$ fixes~$N$ and interchanges the two cases.
We may therefore assume $\Imm R = B$.

Since $\dim B = n^2-n$, one has $\dim\ker R = n$.
Recall that $\Imm R \cdot \ker R,\ker R\cdot \Imm R\subseteq \ker R$.
If some $z\in\ker R$ has a~nonzero entry $z_{rk}$ with $k<n$, then $e_{ir}ze_{kl} = z_{rk}e_{il}$ imply $B\subset\ker R$, a contradiction for $n\geq3$.
Thus $\ker R = K$.
\end{proof}

Below, we denote the identity matrix $e_{11}+\ldots+e_{nn}$ as $E$.
Define
$L = \sum_{i=1}^{n-1}i e_{i+1,i}$.

Since $BK\subset K$ and $KB = 0$, the quotient $M_n(F)/K$ is naturally a~$B$-bimodule:
$b(x + K) = bx + K$ and $(x + K)b = xb + K$.
A derivation from $B$ to this bimodule is an $F$-linear map
$\delta\colon B\to M_n(F)/K$ satisfying
$\delta(bc) = b\delta(c) + \delta(b)c$ for all $b,c\in B$.

\begin{lemma}\label{lem:inner-derivation}
Every derivation $\delta\colon B\to M_n(F)/K$ is inner.
\end{lemma}

\begin{proof}
Write $F^n=F^{n-1}\oplus F$. Relative to this decomposition, every $b\in B$ has the form
$b = \begin{pmatrix}
 A & u\\
 0 & 0
\end{pmatrix}$,
where $A\in M_{n-1}(F)$ and $u\in F^{n-1}$. 
Put $e=\diag(E_{n-1},0)$ and write
$\delta(e) = \begin{pmatrix}P&0\\ v&0\end{pmatrix} + K$.
Since $e^2 = e$, the derivation identity gives
$\delta(e) = e\delta(e) + \delta(e)e$, hence $P = 0$.
Now $b = eb$, so
$\delta(b) = e\delta(b) + \delta(e)b$.
Therefore 
$\delta(b) = \begin{pmatrix} *  & 0\\ vA & 0 \end{pmatrix} + K$.

For $u = 0$, write
$\delta\colon \begin{pmatrix}A&0\\ 0&0\end{pmatrix}
\to \begin{pmatrix}d(A)&0\\ vA&0\end{pmatrix} + K$.
The derivation identity for~$\delta$ shows that
$d(AC) = Ad(C) + d(A)C$ for all $A,C\in M_{n-1}(F)$; hence $d$~is a~derivation of $M_{n-1}(F)$.
It is known that all derivations of the matrix algebra are inner, thus $d(A) = AT - TA$ for some
$T\in M_{n-1}(F)$.

Now put
$x = \begin{pmatrix}0&u\\ 0&0\end{pmatrix}$.
Since $xe = 0$, we have
$0 = \delta(x)e + x\delta(e)$. Using
$\delta(e) = \begin{pmatrix}0&0\\ v&0\end{pmatrix} + K$,
comparison of the upper-left blocks gives the contribution $-uv$. Therefore
$$
\delta(b)
 = \begin{pmatrix}
AT-TA-uv & 0\\
vA       & 0
\end{pmatrix} + K.
$$
Define
$X = \begin{pmatrix}
 T & 0 \\
-v & 0 \end{pmatrix}$ and hence $\delta(b) = bX - Xb + K$.
\end{proof}

\begin{lemma}\label{lem:zero-derivation}
Every derivation $\delta\colon B\to M_n(F)/K$ satisfying $\delta(N) = E + K$ is conjugate, by an inner
automorphism commuting with $N$, to $\delta_0(b) = bL - Lb + K$.
\end{lemma}

\begin{proof}
By Lemma~\ref{lem:inner-derivation}, we
have $\delta(b) = bX - Xb + K$ for some fixed~$X\in M_n(F)$.
The condition $\delta(N) = E + K$ says that
$NX - XN = E$ modulo $K$.
Since $NL - LN = E$ modulo $K$, we may write $X = L + Y$, where $NY - YN\in K$.
Hence the first $n-1$ columns of $NY - YN$ vanish.
This determines the first $n-1$ columns of~$Y$ successively and gives
$Y = f(N)$ modulo $K$ for some $f\in F[t]$.
Since $N^{n-1} = e_{1n}\in K$, we may assume $\deg f\leq n-2$.
Therefore $X = L + f(N)$ modulo~$K$.

To prove the statement, it remains to remove the term $f(N)$ by a conjugation which preserves~$N$.
For this purpose, choose an invertible matrix of the form $U=u(N)$.
Since $U$~is a polynomial in~$N$, it commutes with~$N$, so conjugation by~$U$ fixes~$N$.

We now choose $u$ so that the same conjugation removes $f(N)$. Since
$[N^r,L] = rN^{r-1}$ modulo $K$, one has
$[u(N),L] = u'(N)$ modulo $K$. Hence
$$
U(L+f(N))U^{-1}
 = L + (u'(N)+f(N)u(N))U^{-1} \!\!\! \pmod K.
$$

Thus it is enough to choose
$u(t) = 1 + u_1t + \ldots + u_{n-1}t^{n-1}$ satisfying
$u'(t) + f(t)u(t) = 0$ modulo $t^{n-1}F[t]$.
Such a polynomial exists because its coefficients are determined successively and $\charr F = 0$.
Indeed, writing $f(t) = \sum f_it^i$ and $u_0 = 1$, comparison of the coefficient of $t^r$ gives
$(r+1)u_{r+1} = -\sum_{i=0}^r f_i u_{r-i}$ for $0\leq r\leq n-2$.
Note that $U$~is invertible since $U-E$ is nilpotent.
Thus $U(L+f(N))U^{-1} = L$ modulo~$K$.
\end{proof}

\begin{proof}[Proof of Theorem~\ref{thm:main}a]
By~Lemma~\ref{lem:zero-known}, we may assume that $R(1) = N$.
By~Lemma~\ref{lem:zero-image}, we get $\Imm R = B$ and $\ker R = K$. 
Hence $R$ induces a vector-space isomorphism $M_n(F)/K\to B$. 
Let $\delta\colon B\to M_n(F)/K$ be its inverse, it satisfies the Leibniz identity and $\delta(N) = E + K$.
By~Lemma~\ref{lem:zero-derivation}, we may consider $\delta_0$ instead of~$\delta$.

Let $d_i = (i-1)!$, $D = \diag\{d_1,\ldots,d_n\}$ and
$P(x) = DB_0(D^{-1}xD)D^{-1}$.
Direct substitution in \eqref{eq:B0} gives
$P(E) = N$, $\ker P = K$, $\Imm P = B$, and $P(bL - Lb) = b$ for $b\in B$.
Actually, only the last condition requires some comments.
Put $S = D^{-1}LD = e_{21} + \ldots + e_{n,n-1}$.
Let $b = e_{ij}$ with $1\leq i<n$, $1\leq j\leq n$,
then
$D^{-1}bD = (d_j/d_i)e_{ij}$ and
$$
e_{ij}S-Se_{ij}
 = \begin{cases}
 e_{i,j-1}-e_{i+1,j}, & j>1, \\
 -e_{i+1,1}, & j=1. \end{cases}
$$
To prove $P(e_{ij}L - Le_{ij}) = e_{ij}$ it is sufficient to check that $B_0(e_{ij}S - Se_{ij}) = e_{ij}$ holds for all $1\leq i\leq n-1$ and $1\leq j\leq n$.
By~\eqref{eq:B0}, we have $-B_0(e_{i+1,1}) = e_{i1}$ and $B_0(e_{i,j-1}) - B_0(e_{i+1,j}) = e_{ij}$ for $j>1$; in the latter equality the two finite sums differ only by the first term. Therefore
$B_0(e_{ij}S - Se_{ij}) = e_{ij}$.

Let $\psi$ be the inner automorphism from Lemma~\ref{lem:zero-derivation}.
Put $T = \psi^{-1} R\psi$.
The induced maps $\bar T,\bar P\colon M_n(F)/K\to B$ are isomorphisms, and both have the same inverse $\delta_0$.
Hence $\bar T = \bar P$.
Therefore $T = P$, that is, $\psi^{-1} R\psi = P$.
By the definition of~$P$, it is conjugate to~$B_0$. Hence $R$~is conjugate to~$B_0$.
\end{proof}

\section{Nonzero weight}

The assertion for $n\leq2$ follows from the classification~\cite{BGP}
Hence throughout this section we assume $n\geq3$.

By $L_n(F)$ and $U_n(F)$ we denote the subalgebra of strictly lower-triangular and the subalgebra of strictly upper-triangular matrices in $M_n(F)$, respectively.

We first determine how $A_0 = \ker R^n$ and $A_1 = \ker(R+\id)^{n-1}$ look like.

\begin{lemma}\label{lem:triangular}
Let $R$ be a Rota---Baxter operator of weight~1 and $m_R(X) = X^n(X+1)^{n-1}$.
Up to conjugation by an automorphism or antiautomorphism,
$R(1) = \diag\{0,1,\ldots,n-1\}$,
$\ker R^n = L_n(F)\oplus D_n(F)$ and $\ker(R+\id)^{n-1} = U_n(F)$.
\end{lemma}

\begin{proof}
By Lemma~\ref{lem:nz-known}, we may assume that 
$R(1) = a = \diag\{-s,-s+1,\ldots,0,1,\ldots,r-1,r\}$ 
for some natural $r,s$ such that $r+s+1 = n$.
Since the minimal polynomial of the direct sum of the restrictions $R|_{V_d}$ is their least common multiple, the exponent~$n$ of~$X$ is attained on one $V_d$.
Hence $R|_{V_0}$ is nilpotent of index~$n$ and $V_0\subset A_0$.
Applying Lemma~\ref{lem:diag-known}, we may assume that 
$R(e_{ii}) = \sum_{s>i}e_{ss}$ for $1\leq i<n$, and therefore $a = \diag\{0,1,\ldots,n-1\}$.

Analogously, we conclude that $(X+1)^{n-1}$ is attained on $V_{-1}$ or $V_1$.
Up to conjugation with transpose, we may assume
$V_{-1}\subset A_1$.
Since $A_1$ is a subalgebra, it contains the products of $e_{12},e_{23},\ldots,e_{n-1,n}$ and therefore every $e_{ij}$ with $i<j$.

It remains only to exclude lower-triangular matrix units from $A_1$.
By~Lemma~\ref{lem:hom-known}, $A_0,A_1$ are $L_a$-invariant; by~Lemma~\ref{lem:nz-known}, every $V_d$ is $R$-invariant and hence decomposes as
$V_d = (A_0\cap V_d)\oplus(A_1\cap V_d)$.

Since $a = \diag\{0,1,\ldots,n-1\}$, one has $L_a(e_{ij}) = (i-1)e_{ij}$.
Therefore, for a~fixed~$d$, the matrix units $e_{i,i-d}$ have pairwise distinct eigenvalues with respect to $L_a$.
Hence every $L_a$-invariant subspace of $V_d$ is spanned by a subset of these matrix units.
Thus every matrix unit belongs to exactly one of $A_0,A_1$.
If $e_{ji}\in A_1$ for some $j>i$, then since $e_{ij}\in A_1$, the subalgebra $A_1$ contains
$e_{ij}e_{ji} = e_{ii}$, contrary to $V_0\subset A_0$ and $A_0\cap A_1 = (0)$.
So $A_0 = L_n(F)\oplus D_n(F)$ and $A_1 = U_n(F)$.
\end{proof}

\begin{lemma}\label{lem:shifts}
Let $R$ be a Rota---Baxter operator of weight~1 on $M_n(F)$ satisfying the conclusion of Lemma~\ref{lem:triangular}.
Then, up to conjugation by an invertible
diagonal matrix, one has
\begin{equation}\label{eq:Phi0Phi1}
\Phi_1(e_{ij})
 = \begin{cases}
 e_{i-1,j-1}, & 1 < i < j \leq n,\\
 0, & i = 1 < j \leq n,
 \end{cases} \quad
\Phi_0(e_{ij})
 = \begin{cases}
 e_{i+1,j+1}, & 1 \leq j \leq i < n,\\
 0, & 1 \leq j \leq i = n.
 \end{cases}
\end{equation}
\end{lemma}

\begin{proof}
By Lemma~\ref{lem:hom-known},
$\Phi_1$ is an endomorphism of~$A_1$ satisfying
$\ad_a \Phi_1 = \Phi_1 \ad_a$ and $[L_a,\Phi_1] = -\Phi_1$.
The first equality implies $\Phi_1(e_{ij}) \in V_{i-j}$, the second one gives us
$\Phi_1(e_{ij}) = h_{ij}e_{i-1,j-1}$ for $1<i<j$ and
for some $h_{ij}\in F$ and $\Phi_1(e_{1j}) = 0$.
Since $\Phi_1^k = (R+\id)^kR^{-k}$ on $A_1$, its nilpotency index on $V_{-1}$ is $n-1$; therefore
all adjacent coefficients $h_{i,i+1}$ are nonzero.
Applying conjugation with corresponding diagonal matrix, we may assume that all of them equal to~1.

Indeed, let $D = \diag\{d_1,\ldots,d_n\}$ be invertible
and let $\psi_D\in\Aut(M_n(F))$ be given by
$\psi_D(x) = D^{-1}xD$.
Then the endomorphism corresponding to $R^{(\psi_D)}$ on $A_1$ equals
$\Phi_1^{(\psi_D)} = \psi_D^{-1}\Phi_1\psi_D$, we have used here that $\psi_D(A_1) = A_1$.
Since
$\psi_D(e_{pq}) = (d_q/d_p)e_{pq}$ and
$\psi_D^{-1}(e_{pq}) = (d_p/d_q)e_{pq}$, we obtain
$$
\Phi_1^{(\psi_D)}(e_{i,i+1})
 = h_{i,i+1}\frac{d_{i-1}d_{i+1}}{d_i^2}e_{i-1,i}, \quad 2 \leq i \leq n - 1.
$$
Choose arbitrary $d_1,d_2\neq 0$ and define successively
$d_{i+1} = \frac{d_i^2}{h_{i,i+1}d_{i-1}}$
for $2\leq i \leq n-1$. Since all $h_{i,i+1}$ are nonzero, all $d_i$ are nonzero, so $D$ is invertible.

Further, we apply that $\Phi_1$ is an endomorphism, hence~\eqref{eq:Phi0Phi1} hold for~$\Phi_1$.

By~Lemma~\ref{lem:triangular}, we compute
$\Phi_0(e_{ii}) = e_{i+1,i+1}$ for $i<n$ and
$\Phi_0(e_{n,n}) = 0$.
Since $\ad_a \Phi_0 = \Phi_0 \ad_a$ and $[L_a,\Phi_0] = \Phi_0$, we analogously to the case of~$\Phi_1$ deduce
$\Phi_0(e_{p,p-1}) = b_{p-1}e_{p+1,p}$ for $2\leq p\leq n-1$ and for some $b_{p-1}\in F$ and $\Phi_0(e_{n,n-1}) = 0$.

Let us clarify the values of~$b_q$ with the help of~\eqref{RB}.
For $c\in A_1$ and $d\in A_0$, put $x = R^{-1}(c) = (\Phi_1 - \id)c$ and $y = (R + \id)^{-1}(d) = (\id - \Phi_0)d$.
Then $R(x) = c$, $R(y) = \Phi_0(d)$, and $y + R(y) = d$. Substitution in~\eqref{RB} gives
$$
c\Phi_0(d) = R(\Phi_1(c)d - c\Phi_0(d)).
$$
Taking $c = e_{p,p+1}$ and $d = e_{p,p-1}$, where $2 \leq p \leq n - 1$, we get
$b_{p-1}e_{pp} = R(e_{p-1,p-1} - b_{p-1}e_{pp})$.
Using
$R(e_{ii}) = \sum_{s>i}e_{ss}$ and comparing the coefficient of $e_{pp}$ on both sides, we obtain $b_{p-1} = 1$. Hence
$\Phi_0(e_{p,p-1}) = e_{p+1,p}$ for $2 \leq p \leq n - 1$. We again apply that $\Phi_0$ is an endomorphism on~$A_0$ and prove~\eqref{eq:Phi0Phi1}.
\end{proof}

\begin{proof}[Proof of Theorem~\ref{thm:main}b]
By~Lemma~\ref{lem:nz-known}, we may assume that
$m_R(X) = X^n(X+1)^{n-1}$.
Apply Lemmas~\ref{lem:triangular} and~\ref{lem:shifts}.

On $A_1$, $\Phi_1 - \id = R^{-1}$, hence
$R = -(\id-\Phi_1)^{-1}$. Since $\Phi_1$ is nilpotent, $R(e_{ij}) = -\sum_{r\geq0}e_{i-r,j-r}$ for $i<j$. On $A_0$,
$\Phi_0 = R(R+\id)^{-1}$ gives $R = \Phi_0(\id-\Phi_0)^{-1}$, and therefore
$R(e_{ij}) = \sum_{r\geq1}e_{i+r,j+r}$ for $i\geq j$. These equalities coincide with~\eqref{eq:B1}, so $R = B_1$.
\end{proof}

\section*{Acknowledgements}

The research was carried out within the framework of the Sobolev
Institute of Mathematics state contract (project FWNF-2026-0017).

\noindent Vsevolod Gubarev \\
Sobolev Institute of Mathematics \\
Acad. Koptyug ave. 4, 630090 Novosibirsk, Russia \\
Novosibirsk State University \\
Pirogova str. 1, 630090 Novosibirsk, Russia \\
e-mail: wsewolod89@gmail.com
\end{document}